\documentclass{amsart}
\usepackage{amsfonts,amscd,amssymb,amsmath,amsthm}
\usepackage{graphicx,caption,subcaption,mathrsfs,appendix}
\usepackage{xparse,centernot}
\usepackage[T1]{fontenc}

\usepackage{ctable,color}

\begin{document}

\newtheorem{theorem}{Theorem}[section]
\newtheorem{lemma}[theorem]{Lemma}
\newtheorem{corollary}[theorem]{Corollary}
\newtheorem{conjecture}[theorem]{Conjecture}
\newtheorem{cor}[theorem]{Corollary}
\newtheorem{proposition}[theorem]{Proposition}
\theoremstyle{definition}
\newtheorem{definition}[theorem]{Definition}
\newtheorem{example}[theorem]{Example}
\newtheorem{claim}[theorem]{Claim}
\newtheorem{remark}[theorem]{Remark}

\newenvironment{pfofthm}[1]
{\par\vskip2\parsep\noindent{\sc Proof of\ #1. }}{{\hfill
$\Box$}
\par\vskip2\parsep}
\newenvironment{pfoflem}[1]
{\par\vskip2\parsep\noindent{\sc Proof of Lemma\ #1. }}{{\hfill
$\Box$}
\par\vskip2\parsep}

%%%%%%%%%%%%%%%%%%%%%%%%%%%%%%%%%%%%%%%%%%
%%% General macros
%%%%%%%%%%%%%%%%%%%%%%%%%%%%%%%%%%%%%%%%%%

\newcommand{\R}{\mathbb{R}}
\newcommand{\T}{\mathcal{T}}
\newcommand{\C}{\mathcal{C}}
\newcommand{\G}{\mathcal{G}}
\newcommand{\Z}{\mathbb{Z}}
\newcommand{\Q}{\mathbb{Q}}
\newcommand{\E}{\mathbb E}
\newcommand{\N}{\mathbb N}

\newcommand{\barray}{\begin{eqnarray*}}
\newcommand{\earray}{\end{eqnarray*}}

\newcommand{\beq}{\begin{equation}}
\newcommand{\eeq}{\end{equation}}

%%%%%%%%%%%%%%%%%%%%%%%%%%%%%%%%%%%%%%%%%%
%%% Probability Macros
%%%%%%%%%%%%%%%%%%%%%%%%%%%%%%%%%%%%%%%%%%

\renewcommand{\Pr}{\mathbb{P}}
\newcommand{\as}{\text{a.s.}}
\newcommand{\Prob}{\Pr}
\newcommand{\Exp}{\mathbb{E}}
\newcommand{\expect}{\Exp}
\newcommand{\1}{\mathbf{1}}
\newcommand{\prob}{\Pr}
\newcommand{\pr}{\Pr}
\newcommand{\filt}{\mathscr{F}}
\DeclareDocumentCommand \one { o }
{%
\IfNoValueTF {#1}
{\mathbf{1}  }
{\mathbf{1}\left\{ {#1} \right\} }%
}
\newcommand{\Bernoulli}{\operatorname{Bernoulli}}
\newcommand{\Binomial}{\operatorname{Binom}}
\newcommand{\Beta}{\operatorname{Beta}}
\newcommand{\Binom}{\Binomial}
\newcommand{\Poisson}{\operatorname{Poisson}}
\newcommand{\Exponential}{\operatorname{Exp}}

%%%%%%%%%%%%%%%%%%%%%%%%%%%%%%%%%%%%%%%%%%
%%% Random Graph/Complex Macros
%%%%%%%%%%%%%%%%%%%%%%%%%%%%%%%%%%%%%%%%%%

\newcommand{\link}{\mbox{lk}}
\newcommand{\Deg}{\operatorname{deg}}
\newcommand{\vertexsetof}[1]{V\left({#1}\right)}
\renewcommand{\deg}{\Deg}
\newcommand{\oneE}[2]{\mathbf{1}_{#1 \leftrightarrow #2}}
\newcommand{\ebetween}[2]{{#1} \leftrightarrow {#2}}
\newcommand{\noebetween}[2]{{#1} \centernot{\leftrightarrow} {#2}}
\newcommand{\Gap}{\ensuremath{\tilde \lambda_2 \vee |\tilde \lambda_n|}}
\newcommand{\dset}[2]{\ensuremath{ e({#1},{#2})}}
\newcommand{\EL}{{ L}}
\newcommand{\ER}{{Erd\H{o}s--R\'{e}nyi }}
\newcommand{\zuk}{{\.{Z}uk}}

%%%%%%%%%%%%%%%%%%%%%%%%%%%%%%%%%%%%%%%%%%
%%% Paper-Specific Macros
%%%%%%%%%%%%%%%%%%%%%%%%%%%%%%%%%%%%%%%%%%

\newcommand{\frm}{\ensuremath{ 2\log\log m}}
\DeclareDocumentCommand \fuzz { o o }
{
\IfNoValueTF {#1}
  { \aleph_M }
  { \IfNoValueTF { #2 }
    { \aleph_{M}^{{#1}} }
    { \aleph_{M}^{{#1}}({#2})}
  }
}
\newcommand{\csubzero}{c_{0}}
\newcommand{\csubone}{c_{1}}
\newcommand{\csubtwo}{c_{2}}
\newcommand{\csubthree}{c_{3}}
\newcommand{\csubstar}{c_{*}}
\newcommand{\INE}{I^{\epsilon}}
\newcommand{\rsp}{1-C\exp(-md^{1/4}\log n)}
\newcommand{\lc}{\ensuremath{ \operatorname{light}(x,y)}}
\newcommand{\hc}{\ensuremath{ \operatorname{heavy}(x,y)}}
\DeclareDocumentCommand \pam { O{m} }
{
P_{{#1}}
}
\DeclareDocumentCommand \sca { O{j} }
{ Q_{ {#1} } }
\DeclareDocumentCommand \mga { O{j} }
{ A_{ {#1} } }
\DeclareDocumentCommand \mgb { O{j} O{c} }
{ Q_{ {#1} }^{ {#2} } }
\DeclareDocumentCommand \mgc { O{j} O{c} }
{ U_{ {#1} }^{ {#2} } }
\DeclareDocumentCommand \Filt { O{j} }
{ \mathscr{F}_{{#1}} }

\title{The Power of Choice combined with Preferential Attachment}
\author{Yury Malyshkin}
\address{Department of Mathematics and Mechanics, Moscow State University\\
Laboratory of Solid State Electronics, Tver State University}
\email{yury.malyshkin@mail.ru}
\author{Elliot Paquette}
\address{Department of Mathematics, Weizmann Institute of Science}
\email{paquette@weizmann.ac.il}
\thanks{YM gratefully acknowledges the support of the Weizmann Institute of Science, where this work was performed.
EP gratefully acknowledges the support of NSF Postdoctoral Fellowship DMS-1304057.
}
\date{\today}
\maketitle

\begin{abstract}
We prove almost sure convergence of the maximum degree in an evolving tree model combining local choice and preferential attachment.  
 At each step in the growth of the graph, a new vertex is introduced. 
A fixed, finite number of possible neighbors are sampled from the existing vertices with probability proportional to degree.  Of these possibilities, the vertex with the largest degree is chosen.  The maximal degree in this model has linear or near-linear behavior.  This contrasts sharply with what is seen in the same choice model without preferential attachment.  The proof is based showing the tree has a persistent hub by comparison with the standard preferential attachment model, as well as martingale and random walk arguments.
\end{abstract}

\section{Introduction}
In the present work we further explore how the addition of choice affects the classic preferential attachment model (see~\cite{barabasi,KrReLe}), building on previous work \cite{DSKrM,MP13,KR13}. The preferential attachment graph is a time-indexed inductively constructed sequence of graphs, constructed the following way. We start with some initial graph and then on each step we add a new vertex and an edge between it and one of the old vertices, chosen with probability proportional to its degree. Many different properties of this model have been obtained in both the math and physics literature (see~\cite{barabasi, KrReLe,Mori, DvdHH}). 

In current work we are interested in the degree distribution and in particular in the maximal degree. For the preferential attachment model this problem is studied in~\cite{FFF04,Mori}.  It is shown in \cite{Mori} that the maximum degree $\Delta(t)$ at time $t$ has that $\Delta(t)t^{-1/2}$ converges almost surely to a non-degenerate absolutely continuous distribution.
%, for any function $f$ with $f(t)\rightarrow\infty$ as $t\rightarrow\infty$, $\frac{t^{1/2}}{f(t)}\leq\Delta(t)\leq t^{1/2}f(t)$ with high probability, where $\Delta(t)$ is the highest degree of the preferential attachment graph at time $t$. 
In~\cite{MP13}, limited choice is introduced into the preferential attachment model. More specifically, at each step we independently choose $2$ (or $d$ in general) existing vertices with probability proportional to degree and connect the new vertex with the one with smaller degree. In~\cite{MP13} it is shown that the maximal degree at time $n$ in such a model will be $\log\log n/\log 2 + \Theta(1)$ with high probability ($\log\log n/\log d$ in case of $d$ choices).  There, it is also conjectured by the present authors that if we choose the vertex with the higher degree, the maximal degree will be of order $n/\log n$. Subsequently, this is studied in the physics literature~\cite{KR13}, where the analysis is expanded to show that for $d=2$ this is indeed the case while for $d>2,$ the maximal degree has linear order.

We will give exact first-order asymptotics for the maximal degree in the max-choice model and show almost sure convergence of the appropriately scaled maximal degree.  We now describe the model in more detail.

Define a sequence of trees $\{ \pam \}$ given by the following rule.  Let $P_1$ be the one-edge tree.  Given $P_{m-1},$ define $P_m$ by first adding one new vertex $v_{m+1}$.  Let $X^1_m,\ldots,X^d_m$, where $d\geq 2$, be i.i.d. vertices from $\vertexsetof{ \pam },$ where $\vertexsetof{P}$ is the set of vertices of $P$ chosen with probability
\[
\Pr \left[
X^1 = w
\right] = \frac{\deg w}{2m}.
\]
Note that as the graph has $m$ edges, $\sum_{w}\deg w=2m$. Finally, create a new edge between $v_{m+1}$ and $Y_m,$ where $Y_m$ is whichever of $X^1_m$,...,$X^d_m$ has larger degree. In the case of a tie, choose according to an independent fair coin toss.  We call this the \emph{max-choice preferential attachment tree}.

Let us formulate our main theorem:
\begin{theorem}
\label{thm:max_degree}

In the case $d=2,$ the maximum degree $M_n$ of $P_n$ has
\[
\lim_{n\to\infty} \frac{M_n\log n}{n} = 4
\]
a.s. For $d>2$ the maximum degree $M_n$ of $P_n$ has
\[
\lim_{n\to\infty} \frac{M_n}{n} = x_{\ast}
\]
a.s., where $x_{\ast}$ is the unique positive solution of equation $1-(1-x/2)^{d}=x$ in the interval $0\leq x\leq 2$.
\end{theorem}

Our proof is based on the existence of a \emph{persistent hub}, i.e. a single vertex that in some finite random time becomes the highest degree vertex for all time after. Using this, instead of analyzing the maximum degree over all vertices we effectively only need to analyze the degree of just one vertex.% which becomes the persistent hub.

\begin{proposition}
\label{prop:persistent_hub}
There exists random $N$ and $K$ that are finite almost surely so that at any time $n \geq N$, the vertex $v_K$ has the highest degree among all vertices.
\end{proposition}

Let $L_n$ denote the number of vertices at time $n$ that have maximal degree.  The dynamics of $M_n$ are given by the rule
\begin{equation}
\label{eq:master}
M_{n+1}-M_n = \begin{cases}
1 & \Pr = 1 - \left(1-\frac{M_nL_n}{2n}\right)^{d} \\
0 & \text{else.}
\end{cases}
\end{equation}
The effect of Proposition~\ref{prop:persistent_hub} is that for some $N < \infty$ random and sufficiently large, $L_n = 1$ for all $n > N.$  If we were to assume that $L_n$ were identically one, we effectively consider a simple multi-choice urn.

This urn contains $2$ types of balls, colored black and colored white, with the number of black balls corresponding to $M_n$ and the number of white balls being $2n - M_n.$  
At every time step, $d$ balls are sampled from the urn with replacement and then put back into the urn.  If all are white, then two white balls are added back to the urn.  If at least one is black, then one white ball and one black ball are added to the urn.  Such urn models with multiple samplings have appeared recently in the literature (see~\cite{Kuba13,Chen05}), although this appears to be an uncovered case.
%Our methods could also show that for this $M_n$ as well, the result of Theorem~\ref{thm:max_degree} holds.

%\begin{remark}
%If we instead have an urn process in which one adds $a+b$ balls when both samples are white and $a$ mauve balls and $b$ white balls when at least one is mauve, it is reasonable to inquire if this has the same limiting behavior.  However, it can be seen that the $n/\log n$ growth rate is a particular to the $a=b$ case.  As the $a=b$ case can be simply be identified as a multiple of the $a=b=1$ case, this model is in a sense already as general as possible.
%\end{remark}
%
\subsection*{Proof approach and organization}

We start in section~\ref{sec:apriori} with some initial lower-bound estimates for the maximal degree.  All subsequent arguments require that the maximal degree grows quickly enough to ensure deterministic behavior takes over.
% and will be used in subsequent sections.

In section \ref{sec:hub} we prove the existence of the persistent hub, which allows us to consider the degree of a single vertex instead of the maximal degree. The argument follows the proof of~\cite{Galashin} for convex preferential attachment models and consists of two steps.  First, we show that the number of \emph{possible leaders}, vertices that have maximal degree at some time, is almost surely finite; this follows on account of the maximal degree growing quickly enough that vertices added after a long time have a very small probability of ever catching up.  Second, we show that any two vertices have degrees that change leadership only finitely many times.  These arguments rely heavily on comparison with the preferential attachment model and the P\'olya urn respectively.

In sections \ref{sec:final2} and \ref{sec:finald} we prove convergence of the scaled maximal degree in the cases $d=2$ and $d>2$ respectively, which require different analyses.  From~\eqref{eq:master}, we anticipate the maximal degree $M_n$ of the graph evolves according to the differential equation
\begin{equation*}
  \frac{dM}{dt} =  1-(1-M/2t)^d. 
\end{equation*}
Setting $u(t) = M(e^t)e^{-t},$ we get that $u$ satisfies the autonomous differential equation
\[
  u' + u = 1 - (1-u/2)^d.
\]
In the case $d=2,$ this can be explicitly solved to give $M(t) = 4t/(\log t + C),$ while in the $d>2$ case, we are led to consider critical points, which are solutions of $1-(1-x/2)^{d} = x.$
When $d>2$ there are two solution of the equation $1-(1-x/2)^{d}=x$ in the interval $0\leq x\leq 2,$ but it only has one stable solution $x_{\ast}$ (meaning that $u'$ has the opposite sign of $u-x_{\ast}$ in a neighborhood of $x_{\ast}$).
%function $1-(1-x/2)^{d}$ intersects line $y=x$ from above to below. 

In section \ref{sec:final2} we prove the $d=2$ case by considering explicit scale functions of $M_n$ that can be guessed from the solution of the differential equation.  
%We begin by showing that $M_n/n$ visits any neighborhood of $x_{\ast}$ infinitely often.  Subsequently, we use random walk arguments to show that due to the stability of the point $x_{\ast}$ the process $M_n/n$ converges to $x_{\ast}.$
In section \ref{sec:finald}, we prove the $d >2$ case, which can be formulated generally as follows.  Consider a continuous function $q : [0,1] \to [0,1]$ and define a process $\{T(n), n\geq n_{0}\}$, started from point $T(n_0)=T_0, 0<T_0<n_0$, such that the increments $T(n+1)-T(n)$ are independent $\Bernoulli(q(T(n)/n))$ variables conditioned on $\sigma(T_n).$  This problem has appeared many times in the stochastic approximation literature under the name of the Robbins-Monro model (see~\cite{KushnerBook} or \cite{Benaim}).  Off the shelf techniques are nearly applicable to the situation for $M_n/n,$ but still require that we show that $M_n/n$ are in a neighborhood of $x_{\ast}$ infinitely often, which is the bulk of the work here.  We then give a quick random walk argument to show that $M_n/n$ converges to $x_{\ast}.$ 

%One could apply the technique used in section 6 to more general problem. Consider some function $q(x)$ (with some limitation on it, for example, $q(x)$ is continuous and its number of intersections with line $y=x$ is finite) on interval [0,1] that takes values in interval [0,1]. Consider process $\{T(n), n\geq n_{0}\}$, started from point $T(n_0)=T_0, 0<T_0<n_0$, such that inrcements $T(n+1)-T(n)$ are independ and take value 1 with probability $q(T(n)/n)$ and value 0 otherwise. The question is what is the limit $T(n)/n$? We belive that only points at which function $q(x)$ intersects line $y=x$ from above to below could be the limited points for $T(n)/n$, and each of them has positive probability to be the limit (at least for some $T_{0}$). They distribution depends on starting point (values $T_{0}$ and $n_{0}$) and the function $q(x)$, but we cant say anything about it in general.
\section{Discussion}

Theorem \ref{thm:max_degree} allows us to complete Table~\ref{tab:results} about the influence of choice on the maximum degree of growing random trees. In summary, for the min-choice models, the effect of the choice completely overwhelms the extra effect of the preferential attachment.  On the other hand, the combined effect of preferential attachment with max-choice completely changes the structure of the graph and the order of the maximum degree (see also Figure~\ref{fig:rendering} for a simulation of these trees).  In comparison, adding max-choice to the uniform attachment model produces only a quantitative increase in the maximum degree.

Theorem \ref{thm:max_degree} along with Proposition \ref{prop:persistent_hub} provide us information about the degree sequence of the graph and some structural information about the graph, but it would be nice to know more topological information about the tree. One natural topological property to consider is the diameter of the tree.  

\ctable[
  notespar,
  caption = {Comparison of max/min-choice for $2$ choices with preferential or uniform attachment.   }, 
	label = {tab:results},
	pos = t
]{r | c c c}{
  \tnote[(a)]{ \cite{Mori}}
  \tnote[(b)]{ \cite{MP13}}
  \tnote[(c)]{ \cite{DSKrM}}
  \tnote[(d)]{ To our knowledge, this is not claimed formally anywhere.  However, getting the correct order is an elementary exercise. }
}{ \FL
& max-choice & no-choice & min-choice \ML
Preferential attachment & 
$\frac{4n}{\log n}(1+o(1))$ & 
$\Theta(n^{1/2})$ \tmark[(a)]& 
$\frac{\log\log n}{\log 2} + \Theta(1)$ \tmark[(b)] \NN
Uniform attachment & 
$O(\log n)$ \tmark[(c)]& 
$O(\log n)$ \tmark[(d)] & 
$O(\log\log n)$ \tmark[(c)] \LL
}

In the standard preferential attachment model the diameter is known to be logarithmic~\cite{DvdHH}.  It is natural to wonder if the diameter in this situation is smaller. To increase the diameter we must add an edge between a new vertex and an existing vertex of degree-1. In the max-choice model, choosing such a vertex is still not too rare; for while it is less likely to choose a degree-1 vertex than in preferential attachment, there are $\Theta(n)$ degree one vertices.  Thus, degree-1 vertices are selected at each time step with some probability bounded away from $0.$  Conditional on choosing a vertex of degree $1,$ the exact choice of vertex is uniform over all possible choices.  Thus we conjecture the diameter of the graph grows at a rate that is commensurate to that of the preferential attachment model.  

The rate could be different if we change the rule of picking the vertex in the case of a tie.  The model we study breaks ties uniformly, but in fact any tie breaking rule have the same degree sequence evolutions in law.  However, it could significantly affect the structure of the graph. For example, if instead of a fair coin toss we define a function $\operatorname{rad}(v_{j})=\max_{i}(dist(v_{i},v_{j}))$, and on each step we choose the vertex with the smallest value of $\operatorname{rad}(v_{i})$ among all vertices with the same degrees, we anticipate something like order $\log\log n$ diameter (see also~\cite{KrReLe}, where such a model is considered).
%On the other hand, considering the distance of vertices to the maximal degree vertex, we expect the choice of max-degree to be very nearly the same as the choice of smaller distance.  In this case, the distances to the hub evolve very nearly as a type of balls-in-bins model with choice: i.e. at every step a new ball is added to a random bin, with the bin chosen from some collection of random bins according to which bin has the smallest load.  Such a model is known to have max-load order $\log \log n / \log d.$
%Doing so we could find ourself in situation close with analise of maximum degree in case of min-choices, in witch case the result is $\log\log n/\log d$. There is still two differences from this model. First, in min-choices case we choose vertex with highest degree with probability, proportional to it degree, but it should not affect asymptotic behavior since in bins and balls model with d choices asymptotic is the same. Second difference could possibly increase the $\log\log n/\log d$ asymptotics. In min-choices case there were not much vertices with maximum degree. When we study the diameter there is a lot of vertices with degree 1, and when we choose exact vertex among vertices of the same degree we follow uniform law. So $d$ choices affected only probability to choose vertex with degree 1.

In the model we study here, we consider only graphs that are trees, and we believe that similar results should hold for classes on non-tree models.  One such natural model would be to add more than one edge at each step.  A second would be to flip a coin at each time step to choose between adding a new vertex or adding an edge between existing vertices with probability.  If adding a vertex, the rule would be the same as in our model, while for adding an edge there are a few natural possibilities that could affect structure of the graph. Here is one of such rules. We choose the first vertex with probability proportional to the degrees of the vertices of the graph (which is preferential attachment without choice), and then we choose the second vertex among all non-adjacent vertices using the max-$d$ choice rule. In this case the degree distribution we anticipate max-degree behavior to match the tree model.  Note that both these methods will only increase the average degree of the vertices of the graph.

\begin{figure}[h]
\centering
	\begin{subfigure}[t]{0.30\textwidth}
	\includegraphics[width=\textwidth]{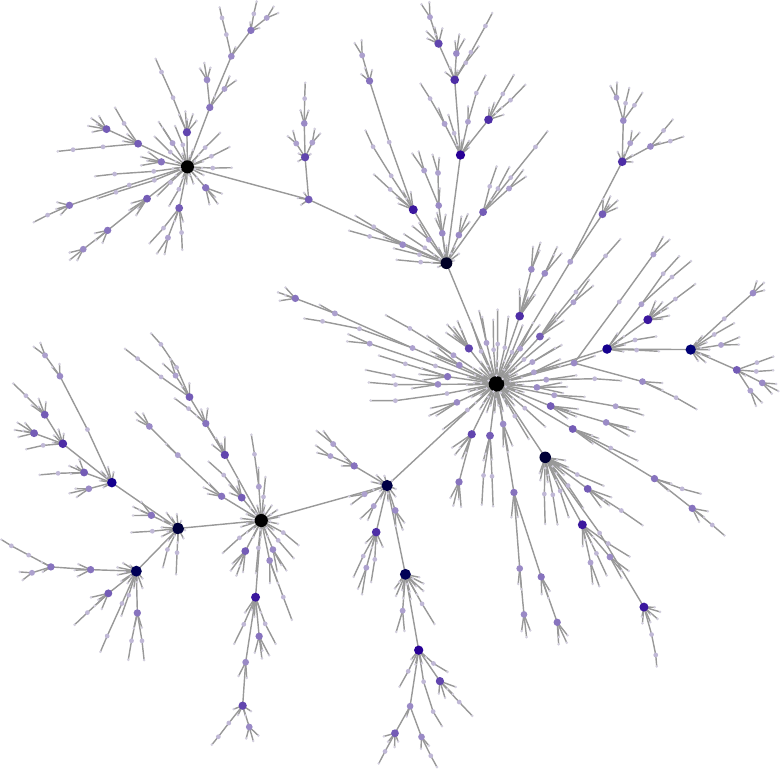}
	\caption{The preferential attachment tree.}
	\end{subfigure}
	\begin{subfigure}[t]{0.30\textwidth}
	\includegraphics[width=\textwidth]{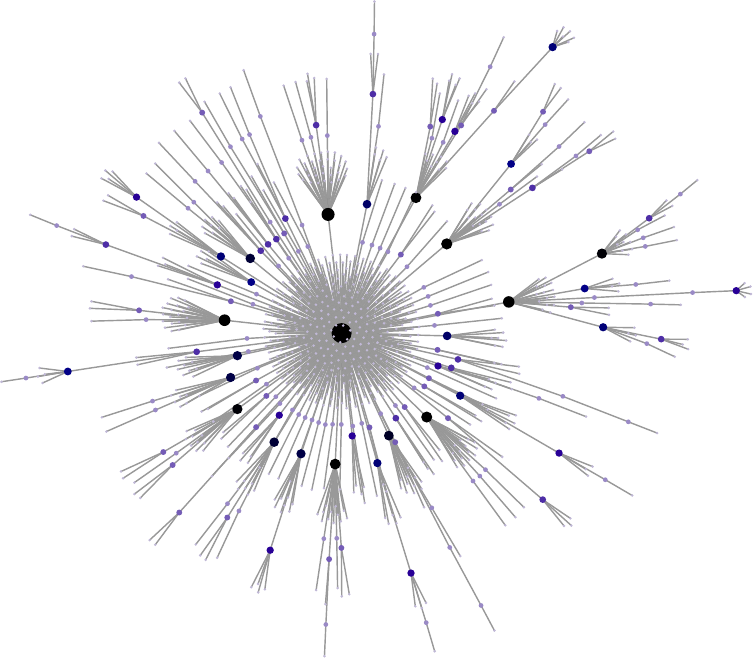}
	\caption{The max-choice preferential attachment tree.}
	\end{subfigure}
	\begin{subfigure}[t]{0.30\textwidth}
	\includegraphics[width=\textwidth]{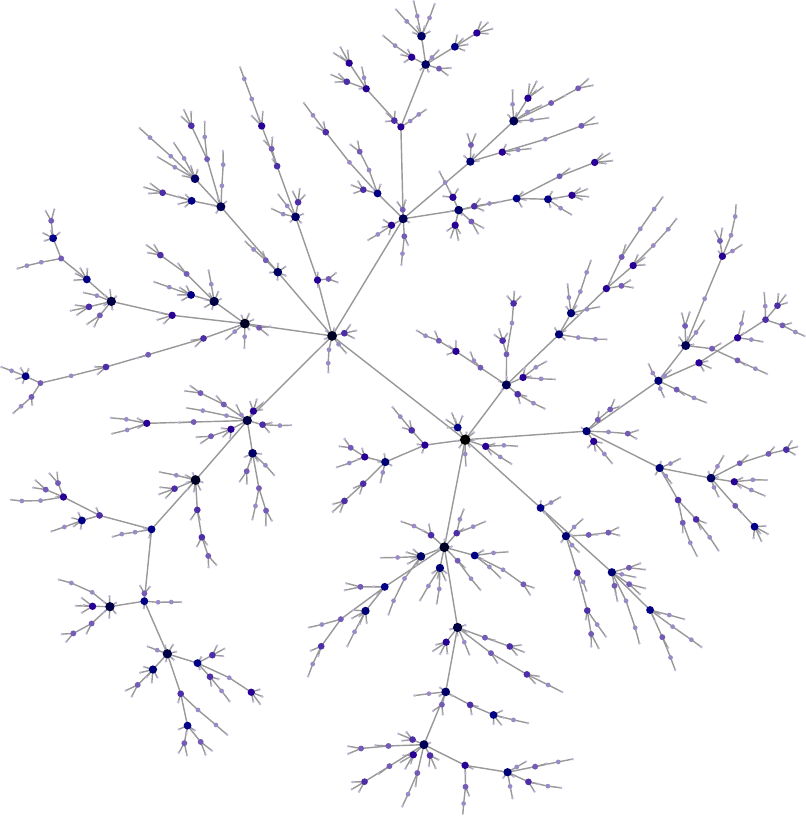}
	\caption{The max-choice uniform attachment tree.}
	\end{subfigure}
	\caption{All renderings are with $1000$ vertices.}
	\label{fig:rendering}
\end{figure}

\section{\emph{A priori} estimates}
\label{sec:apriori}
We begin with a pair of lower bounds for the growth of the maximal degree.  These are needed both for the persistent hub proof and the eventual precise estimates.
We will frequently use the following lemma of~\cite{Galashin}.
\begin{lemma}
\label{lem:numbers}
Suppose that a sequence of positive numbers $r_n$ satisfies
\[
r_{n+1} = r_n\left(1+\frac{\alpha}{n+x}\right),~n \geq k
\]
for fixed reals $\alpha > 0,$ $k > 0$ and $x.$  Then $r_n/n^{\alpha}$ has a positive limit.
\end{lemma}
This is easily checked from a direct computation.
We will use $\Filt$ denote the natural filtration for the whole tree, i.e.
\(
\Filt = \sigma( \pam[1], \pam[2], \ldots, \pam[n]).
\)
With respect to this filtration, both $M_n$ and $L_n$ are measurable.
\begin{lemma}
\label{lem:starting_low_bound}
With probability $1,$
\[
\inf_{n} M_n/n^{3/8} > 0.
\]
\end{lemma}
\begin{proof}
Define $C_{n+1}=\frac{8n}{8n-3}C_n=(1+\frac{3}{8n-3})C_n,$ with $C_1 = 1.$  By Lemma~\ref{lem:numbers} we have that $C_nn^{-3/8}$ converges to a positive limit.  Now, we will show that $C_{n}/M_{n}$ is a supermartingale from which the desired conclusion follows.

Let $p_{n}$ be the probability to increase maximum at the $n^{th}$ step.
Note that
\begin{align*}
p_{n} &=1 - \left(1-\frac{M_nL_n}{2n}\right)^{d} 
\geq 1-\left(1-\frac{M_{n}}{2n}\right)^{d}  \\ 
&\geq 1-\left(\frac{2n-M_{n}}{2n}\right)^{2} 
\geq \frac{M_{n}}{n}-\frac{M_{n}^2}{4n^2}   \\
&= \frac{M_n}{n}\frac{4n-M_n}{4n}\geq\frac{3M_{n}}{4n}.
\end{align*}
For $1/M_{n}$ we get
\begin{align*}
\E\left[1/M_{n+1} | \Filt \right]
&=\frac{p_{n}}{M_{n}+1}+\frac{1-p_{n}}{M_{n}} 
=\frac{M_n+1-p_n}{M_n(M_n+1)} \\
&=\frac{1}{M_n}\left(1-\frac{p_n}{M_n+1}\right) 
\leq\frac{1}{M_{n}}\left(1-\frac{p_n}{2M_n}\right) \\
&\leq \frac{1}{M_n}\left(1-\frac{3}{8n}\right).
\end{align*}
\end{proof}

We will now show that with this initial argument, it is possible to improve the result by an application of the same argument.

\begin{lemma}
\label{lem:improved_low_bound2}
For any fixed $\delta>0$,
\[
\liminf_{n\to\infty} M_n/n^{3/4-\delta} = \infty
\]
a.s.
\end{lemma}
\begin{proof}
Let $\tau_{\epsilon}$ be the stopping time given by
\[
\tau_{\epsilon} = \inf\{
n~:~M_n < \epsilon n^{3/8}
\}.
\]
From Lemma~\ref{lem:starting_low_bound}, we have that
\(
\Pr \left[
\tau_{\epsilon} < \infty
\right] \to 0
\)
as $\epsilon \to 0.$
Set $O_{\epsilon}$ to be the event $\{\tau_{\epsilon} = \infty\}.$

As in the proof of Lemma~\ref{lem:starting_low_bound}, we get that
\(
p_{n}\geq \frac{3M_{n}}{4n}.
\)
Then for $1/M_{n+1},$ it holds that
\[
\E(1/M_{n+1}|\filt_{n})
=\frac{1}{M_n}\left(1-\frac{p_n}{M_n+1}\right)
\leq\frac{1}{M_n}\left(1-\frac{3}{4n}\frac{M_n}{M_n+1}\right).
\]
For each fixed $\delta>0$ and $n<\tau_\epsilon,$
\[
\frac{M_{n}}{M_{n}+1}
=1-\frac{1}{M_{n}+1}
\geq 1-\frac{1}{1+\epsilon n^{3/8}}
\geq 1-\frac{4\delta}{6}
\]
if $n>n_{0}$ for some sufficiently large $n_{0}=n_{0}(\delta,\epsilon)$.
Hence for $\tau_{\epsilon} > n>n_{0}$ we get
\[
\E(1/M_{n+1}|\filt_{n})
\leq \frac{1}{M_n}\left(1-\frac{3/4-\delta/2}{n}\right).
\]
Define $R_{n+1}=\frac{4n}{4n-3+2\delta}R_n\geq(1+\frac{3/4-\delta/2}{n})R_n,$ $n\geq n_{0}$. Then $R_{n}/M_{n}$ is a supermartingale and from Lemma~\ref{lem:numbers} it follows that $R_{n}n^{-(3/4-\delta/2)}$ converges to a positive finite limit. Setting $A_n = R_n/M_n,$ we have that by Doob's theorem $A_{n\wedge\tau_{\epsilon}}$ tends to a finite limit with probability 1.  Hence, conditioned on $O_{\epsilon},$ we have that $M_n/n^{3/4-\delta} \to \infty~\as$  Thus, it follows that
\[
\Pr\left[
\liminf_{n\to\infty} M_n/n^{3/4-\delta} = \infty
\right]
\geq \Pr\left[
\{\liminf_{n\to\infty} M_n/n^{3/4-\delta} = \infty\}
\cap O_\epsilon
\right]
=\Pr\left[ O_\epsilon \right].
\]
Taking $\epsilon \to 0,$ we conclude the proof.
\end{proof}

\section{Persistent hub}
\label{sec:hub}

Our method of proof is essentially by comparison with the preferential attachment model, and we use the machinery of~\cite{Galashin} developed for this task.  First we estimate the probability that the degree of the vertex added on the $(k+1)^{st}$ step could exceed the degree of vertex with highest degree at step $k$. For this we use the following lemma:
\begin{lemma}
\label{lem:change_of_leader}
The probability $\pi(k)$ that the degree of the vertex added on the $k$-th step becomes maximal does not exceed
\[
\pi(k)\leq\frac{P(M_{k})}{2^{M_{k}}},
\]
where $P(A)$ is some polynomial of $A$ and $M_{k}$ is the maximum degree at the $k$-th step.  Hence, the number of vertices that at some point in the process have maximal degree is finite almost surely.
\end{lemma}
First we prove the following auxiliary result:
\begin{lemma}
\label{lem:walk_domination}
Fix $m_0 > 0.$  Let $T_n = (A_n, B_n)$ for $n \geq m_0$ denote the random walk on $\Z^2$ started from $(A_{m_0},B_{m_0})$ that moves one step right or one step up with probabilities proportional to $A_n$ and $B_n$ respectively.
For any pair of vertices $v_{i}$ and $v_{j}$, the probability that their degrees become equal at some time $n\geq m_{0}$ is bounded above by the probability that the random walk $T_{n}=(A_{n},B_{n})$ reaches the line $y=x$,
where $(A_{m_0},B_{m_0}) = (\deg(v_i),\deg(v_j))$ at time $m_0.$ 
%Here $A_{m_{0}}$ and $B_{m_{0}}$ equal to the degrees of $v_{i}$ and $v_{j}$ at time $m_{0}$ and the walk $T$ moves either one step right or one step up with probabilities proportional to $A_{n}$ and $B_{n}$.
\end{lemma}

\begin{proof}
Consider the two-dimensional random walk $S_{n}=(w_{n},u_{n}),$ where $w_{n}$ is the degree of vertex $v_{i}$ and $u_{n}$ is the degree of vertex $v_{j}$. Without loss of generality assume that $w_{m_{0}}>u_{m_{0}}$. We want to show that
\[
  \Pr[ \exists n \geq m_0 : w_n = u_n ] 
  \leq
  \Pr[ \exists n \geq m_0 : A_n = B_n ]. 
\]
We will show the existence of an appropriate coupling of $S_n$ and $T_n.$
To this end, set
\begin{align*}
  F_{n}&=\sum_{v_{k}\in V}\deg v_{k}\one[\deg v_{k} < \deg v_{i}] \text{ and } \\ 
  G_{n}&=\sum_{v_{k}\in V}\deg v_{k}\one[\deg v_{k} \leq \deg v_{j}],
\end{align*}
and let $p_n^w = \Pr[ w_{n+1} = w_n +1]$ and $p_n^u = \Pr[u_{n+1} = u_n + 1].$

The probability that $w_n = \deg v_i$ increases is at least the probability that $v_{i} \in \{X^1_m$,...,$X^d_m\}$ and that all the other choices have degree strictly less than $\deg v_i.$ Thus
\[
  p_n^w \geq \left(\frac{F_n+w_n}{2n}\right)^d - \left( \frac{F_n}{2n}\right)^d.
\]
Likewise, the probability that $u_n = \deg v_j$ increases is at most the probability that vertex $v_{j} \in \{X^1_m$,...,$X^d_m\}$ and $\deg v_j = \max_{1\leq k \leq d} \deg X^k_j.$  Thus
\[
  p_n^u \leq \left(\frac{G_n}{2n}\right)^d - \left( \frac{G_n - u_n}{2n}\right)^d.
\]
So long as $w_n = \deg v_i > \deg v_j = u_n,$ we have $F_n \geq G_n.$  Hence
\begin{align*}
  \frac{p_n^w}{p_n^u} 
  &\geq \frac{ (F_n+w_n)^d - (F_n)^d}{(G_n)^d - (G_n - u_n)^d}  \\
  &\geq \frac{ (G_n+w_n)^d - (G_n)^d}{(G_n)^d - (G_n - u_n)^d}  \\
\intertext{ Using the convexity of $x^d$, we have the bound $|x+y|^d \geq x^d + d x^{d-1}y$ for $x \geq 0.$  Applying this to the previous inequality, we get: }
  \frac{p_n^w}{p_n^u} 
  &\geq \frac{ d (G_n)^{d-1} w_n}{d(G_n)^{d-1}u_n}
  =\frac{ w_n }{ u_n }.
\end{align*}
Thus, 
\[
  \frac{p_n^w}{p_n^w + p_n^u}
  =\frac{1}{1 + \tfrac{p_n^u}{p_n^w}}
  \geq\frac{1}{1 + \tfrac{u_n}{w_n}}
  =\frac{w_n}{w_n + u_n}.
\]
Letting $\tau_1,\tau_2,\tau_3,\ldots$ be the times at which $S_n$ moves, we have that $S_{\tau_n}$ and $T_n$ can be coupled in such a way that both $w_{\tau_n} \geq A_n$ and $u_{\tau_n} \leq B_n$ until the first time $w_{\tau_n} = u_{\tau_n}.$  Thus if at some finite time $w_n = u_n,$ it must also be that there is a time $m \leq n$ at which $A_m = B_m,$ completing the proof.

\end{proof}
The walk $T_{n}$ would describe the evolution of the degrees of two vertices in the preferential attachment model without choices. Hence we can apply to it some of the results from~\cite{Galashin}. We will now use it to prove Lemma \ref{lem:change_of_leader}.
\begin{proof}
Consider the vertex $v_{k+1}$ added on the $k$-th step. Its degree at time $k+1$ equals to 1. Let $A_{k+1}=M_{k+1}$, $B_{k+1}=1$, \text{and} $m_{0}=k+1$.
Corollary 15 of~\cite{Galashin} gives the following estimate for the probability $q(M_{k+1})$ that the walk $T_{n}$, $n>k+1$ moves from the point $(M_{k+1},1)$ to the diagonal:
$$q(M_{k+1})\leq\frac{P(M_{k+1})}{2^{M_{k+1}}},$$
where $P(M_{k+1})$ is some polynomial.

By Lemma~\ref{lem:starting_low_bound} we get that
\(
M_n\geq M n^{3/8}
\)
for some random $M>0$ almost surely.  In particular, $\pi(k)$ forms a convergent series with probability $1$, and by Borel-Cantelli, the number of $k$ for which the vertex added at the $k$-th step have maximal degree at some point in time is finite almost surely.
\end{proof}

 To complete the proof of \ref{prop:persistent_hub} we now need the following lemma:
\begin{lemma}
\label{lem:change_of_leadership}
Consider two vertices that at some time have maximal degree.  With probability $1$ there are only a finite number of times when these vertices have the same degree and are maximal.
\end{lemma}
\begin{proof}
  Let $v_i$ and $v_j$ be two vertices that at some point have equal, maximal degree, and let $m_0$ be the first time that this occurs.
  Consider a two-dimensional random walk $S$ with coordinates equal to $(\deg v_i, \deg v_j)$ for all time $n \geq m_0.$ They have the same degree if and only if the walk is on the line $y=x$.  As in Lemma~\ref{lem:walk_domination}, the probability that $S$ hits the line $y=x$ when started off the line is bounded from above by the probability that $T$ hits the line $y=x.$  Hence the number of times $n \geq m_0$ that $S$ returns to the line $y=x$ is bounded above by the number of times $T$ returns to the line $y=x.$

%Note that since one of the vertices has maximal degree for infinitely many steps (otherwise we consider only finitely many times and the statement of the lemma is trivial), the walk $S$ moves infinitely many times, so by time changing the process, we may consider the walk conditioned to move every time step.

It is a standard fact about the P\'olya urn that if $T_n=(A_n,B_n)$ starts from a point $(t,t)$, then the fraction $A_{n}/(A_{n}+B_{n})$ tends in law to a random variable $H(t)$ as $n$ tends to infinity, where $H(t)$ has a beta probability distribution:
$$H(t)\sim\Beta(t,t).$$
(See also Proposition 16 of~\cite{Galashin})
Since the beta distribution is absolutely continuous, the fraction $A_{n}/(A_{n}+B_{n})$ tends to an absolutely continuous probability distribution for any starting point of the process $T$.  Thus the limit of $A_n/(A_n + B_n)$ exists almost surely, and it takes value $1/2$ with probability 0.  Hence this fraction can be equal to $1/2$ only finitely many times, and so $T$ can return to the line $y=x$ only finitely many times.

Thus, the only way that there can be infinitely many times for which $\deg v_i = \deg v_j$ is if both $\deg v_i$ and $\deg v_j$ stabilize, i.e. there is a $D$ not depending on $n$ and an $n_0$ for which $\deg v_i = \deg v_j = D$ for all $n \geq n_0.$  However, in this case, these degrees are only maximal for finitely many times as the maximal degree goes to infinity by Lemma~\ref{lem:starting_low_bound}, which completes the proof.
\end{proof}

\begin{proof}[Proof of Proposition \ref{prop:persistent_hub} ]
  From Lemma~\ref{lem:change_of_leader} the number of vertices that at some point have maximal degree is finite almost surely, and from Lemma~\ref{lem:change_of_leadership} these finitely many vertices only change leadership finitely many times almost surely.  Thus, after some sufficiently long time, a single vertex remains the maximal degree vertex for all subsequent time.
\end{proof}

\section{ The case d=2 }
\label{sec:final2}
In this section, we show the limiting behavior of the maximum degree in the case $d=2.$ 
From Proposition~\ref{prop:persistent_hub} it follows that
\[
\lim_{C\rightarrow\infty}\Pr[L_{n}=1,\;\forall n\geq C]=1.
\]
Introduce events $D(C)=\{L_{n}=1, \;\forall n\geq C\}$, and the stopping times
\(
\eta_C=\inf_{n \geq C}\{n:L_{n} > 1\}.
\)
%Note that this is not a stopping time.  We also define
%$$\mu_{C}=\begin{cases}
%\eta, & \eta\leq C\\
%\infty, & \eta>C.
%\end{cases}$$
%Prove for $d>2$ will be given in the next section.
For fixed $c > 0$ we define the following set of scale functions of $M_n.$
\begin{equation}
\label{eq:mgles}
\begin{aligned}
\mgb[n] &= \exp( cn / M_n ) / n \\
\mgc[n] &= n\exp( -cn / M_n ).
\end{aligned}
\end{equation}

\begin{lemma}
\label{lem:mgs}
In the following, let $\epsilon > 0$ and $C>0$ be a fixed positive number.
\begin{enumerate}
\item
For each $c<4,$ there is a constant $n_1 = n_1(C,c,\epsilon) \geq C$ sufficiently large so that if
\(
\tau_{\epsilon} = \inf_{n > n_1} \{ n ~:~ M_n < \epsilon n^{0.67} \}
\)
then
\(
\mgb[n \wedge \tau_\epsilon \wedge \eta_C]~n \geq n_1
\)
is a supermartingale.
\item
For each $c>4,$ there is a constant $n_2 = n_2(C,c,\epsilon) \geq C$ sufficiently large so that if
\(
\tau_{\epsilon} = \inf_{n > n_2} \{ n ~:~ M_n < \epsilon n^{0.67} \}
\)
then
\(
\mgc[n \wedge \tau_\epsilon \wedge \eta_C]~n \geq n_0
\)
is a supermartingale.
\end{enumerate}
\end{lemma}
\begin{proof}[Proof of Lemma~\ref{lem:mgs}]
Since we only consider $n \leq \eta_C$ we have that $L_n = 1$ almost surely, and hence
\(
p_n = M_n/n(1-M_n/4n)
\)
for the probability at the $n$-th step that $M_n$ increases.

\noindent \emph{Proof of (i)} We must estimate $\Exp[ \mgb[n+1] \vert \Filt ]$ for $c < 4$ under the assumption that $M_n \geq \epsilon n^{0.67}.$  As we wish to show this is a supermartingale, it suffices to show that there is a $n_0$ sufficiently large so that under these assumptions
\[
\Exp\bigl[ \tfrac{\mgb[n+1]}{\mgb[n]} \vert \Filt \bigr] \leq 1.
\]
The proof follows by Taylor expansion.
\begin{align*}
\Exp\bigl[ \tfrac{\mgb[n+1]}{\mgb[n]} \vert \Filt \bigr]
&= \frac{n}{n+1}\left[
e^{\left( \frac{c}{M_n}\right)}(1-p) + pe^{\left(c\tfrac{n+1}{M_n+1} - \tfrac{cn}{M_n}\right)}
\right] \\
&= 1-\frac{1}{n} +\frac{c}{M_n} +cp\left(\frac{-1}{M_n} + \frac{M_n - n}{M_n(M_n+1)}\right)+O\left(\frac{1}{M_n^2} + \frac{n^2p}{M_n^4}\right).
\intertext{ Noting that $p\leq M_n/n$ and that under our assumption, $M_n = \omega(j^{2/3}),$ it follows that this error term is $o(1/n).$  Substituting in the definition of $p,$ we get}
\Exp\bigl[ \tfrac{\mgb[n+1]}{\mgb[n]} \vert \Filt \bigr]
&= 1-\frac{1}{n} +\frac{c}{M_n}  -c\left(\frac{n+1}{n(M_n+1)}\right)\left(1-\frac{M_n}{4n}\right)+O\left(\frac{1}{n^{1.001}}\right). \\
&\leq 1-\frac{1}{n} +\frac{c}{4n}+O\left(\frac{1}{n^{1.001}}\right).
\end{align*}
Note that constant in the $O(\cdots)$ term depends only on $\epsilon$ and $c.$  Hence, we may find an constant $n_0>{C}$ sufficiently large so that this is always strictly less than $1,$ which completes the proof.

\noindent \emph{Proof of (ii)}  This is precisely the same calculation as was done for (i).  Once more, it suffices to show that for $c> 4,$
\[
\Exp\bigl[ \tfrac{\mgc[n+1]}{\mgc[n]} \vert \Filt \bigr] \leq 1.
\]
If we expand this expectation, we get
\[
\Exp\bigl[ \tfrac{\mgc[n+1]}{\mgc[n]} \vert \Filt \bigr]
= \frac{n+1}{n}\left[
e^{\left( \frac{-c}{M_n}\right)}(1-p) + pe^{\left(-c\tfrac{n+1}{M_n+1} + \tfrac{cn}{M_n}\right)}
\right].
\]
The same calculus shows that we have
\[
\Exp\bigl[ \tfrac{\mgc[n+1]}{\mgc[n]} \vert \Filt \bigr] =
1 + \frac{1}{n} - \frac{c}{4n} + O\left( \frac{1}{n^{1.001}} \right),
\]
so that when $c>4,$ the desired claim holds.
\end{proof}
Using the a priori estimates, we are able to use $\mgb[n]$ to prove the main theorem for $d=2$.
\begin{proof}[Proof of Theorem~\ref{thm:max_degree}]
Using these supermartingales, the proof proceeds along similar lines as in Lemma~\ref{lem:improved_low_bound2}.
Once again set $O_{\epsilon}$ to be the event $\{\tau_{\epsilon} = \infty\}.$  From Lemma~\ref{lem:improved_low_bound2} we have
\[
\liminf_{n\to\infty} M_n/n^{0.67} = \infty~\as
\]
Hence, we have that
\[
\lim_{\epsilon \to 0} \one[\inf_{n > 0} M_n/n^{0.67} \leq \epsilon] = 0~\as
\]
Thus,
\(
\lim_{\epsilon \to 0} \Pr[O_\epsilon] = 1.
\)

On the event $O_{\epsilon}\cap D_{C},$ we have by positive supermartingale convergence that there is some large $R_\epsilon$ random so that
\[
\sup_{n > 0} \mgb[n] < R_{\epsilon} < \infty.
\]
Hence, on this event,
\[
M_n \geq \frac{cn}{\log n + \log R_\epsilon},
\]
and so
\[
\liminf_{n \to \infty} \frac{M_n \log n}{n} \geq c.
\]
Thus we have that
\[
\Pr \bigl[
\bigl\{\liminf_{n \to \infty} \tfrac{M_n \log n}{n} \geq c\bigr\}
\cap O_{\epsilon}
\cap D_{C}
\bigr]
= \Pr \left[ O_\epsilon \cap D_{C}\right],
\]
and so taking $\epsilon \to 0$ and $C \to \infty$ we have that
\[
\liminf_{n \to \infty} \frac{M_n \log n}{n} \geq c~\as
\]
As this holds for any $c<4,$ we conclude the desired lower bound.

The upper bound follows by the exact same machinery.
On the event $O_\epsilon\cap D_{C},$ we have by positive supermartingale convergence that there is some large $R_\epsilon$ random so that
\[
\sup_{n > 0} \mgc[n] < R_{\epsilon} < \infty.
\]
Hence, on this event,
\[
M_n \leq \frac{cn}{\log n - \log R_\epsilon},
\]
and so
\[
\limsup_{n \to \infty} \frac{M_n \log n}{n} \leq c.
\]
Thus we have that
\[
\Pr \bigl[
\bigl\{\limsup_{n \to \infty} \tfrac{M_n \log n}{n} \leq c\bigr\}
\cap O_{\epsilon}
\cap D_{C}
\bigr]
= \Pr \left[ O_\epsilon \cap D_C \right],
\]
and so taking $\epsilon \to 0$ and $C \to \infty$ we have that
\[
\limsup_{n \to \infty} \frac{M_n \log n}{n} \leq c~\as
\]
As this holds for any $c>4,$ the proof is complete.
\end{proof}

\section{case d>2}
\label{sec:finald}

The case $d>2$ requires different analysis from the case $d=2$. Let $x_{\ast}$ be the solution of equation $1-(1-x/2)^{d}=x$ in the interval $(0,2)$. Note that by monotonicity and continuity of each side of the equation, this solution exists and is unique.  From section~\ref{sec:final2}, recall the events $D(C)=\{L_{n}=1, \;\forall n\geq C\}$, and the stopping time
\(
\eta_C=\inf_{n \geq C}\{n:L_{n} > 1\}.
\)
\begin{lemma}
\label{lem:interval_est}
Conditional on $D(C),$
for any $n_{0}>C$ and $\epsilon>0$  there is an $n_1>n_{0}$ random with $n_1$ finite almost surely so that $x_{\ast}-\epsilon<M_{N}/N<x_{\ast}+\epsilon$.
\end{lemma}
\begin{proof}

%We divide proof on two parts.
%Note that to prove that for any $n_{0}$ and $\epsilon>0$ with probability 1 there is $n\geq n_{0}$, such that $x_{\ast}-\epsilon\leq M_{n}/n\leq x_{\ast}+\epsilon$ 
The statement of the lemma is equivalent to the statement that for any $n_{0}$ and $\epsilon>0$ there is $n_{1}\geq n_{0}$ and $n_{2}\geq n_{0}$ such that $x_{\ast}-\epsilon\leq M_{n_{1}}/n_{1}$ and $M_{n_{2}}/n_{2}\leq x_{\ast}+\epsilon$; as the process has bounded increments, if such $n_1$ and $n_2$ exist, there must be a time in between that satisfies the statement of the lemma, provided $n$ is taken sufficiently large.
%Also, without loss of generality we could assume than $n>\eta$.

Recall that $p_n$ is the probability that $M_{n+1} = M_n + 1$ conditional on $\filt_n.$ Note that for $n$ with $C \leq n \leq \eta_C,$
\[
p_{n}
=1-\left(1-\frac{M_{n}}{2n}\right)^{d}
=\frac{M_{n}}{2n}\left(\sum_{i=0}^{d-1}(1-\frac{M_{n}}{2n})^{i}\right).
\]
Hence if we define the function
\[
f(x)=\frac 1 2 \sum_{i=0}^{d-1}(1-x/2)^{i},
\]
then $\frac{p_{n}}{M_{n}}=\frac{1}{n}f(\frac{M_{n}}{n})$. If $x\neq 0$ this function is equal to $\frac{1-(1-x/2)^{d}}{x}$. Therefore $x_{\ast}$ is the solution of equation $f(x)=1$ in the interval $(0,1)$. 
Note that for any $\epsilon>0$ there is a $\delta>0$ so that $f(x)>1+\delta$ if $0\leq x\leq x_{\ast}-\epsilon$ and $f(x)<1-\delta$ if $x_{\ast}+\epsilon\leq x\leq 1$. 

We will start by proving the lower bound.  
Assume that for $n_{0},$ $x_{\ast}-\epsilon>M_{n_{0}}/n_{0}$ (otherwise we could just put $n_{1}=n_{0}$). Let $\phi_{1}$ be the first moment after $n_{0}$ such that $x_{\ast}-\epsilon\leq M_{\phi_{1}}/\phi_{1}$. We need to prove that conditional on the event $\eta_C = \infty,$ $\phi_{1}<\infty$. 
Consider the expectation
\begin{align*}
\Exp\left(\frac{M_{n}}{M_{n+1}} \mid \filt_{n}\right)
&=p_{n}\frac{M_{n}}{M_{n}+1}+1-p_{n} 
=p_{n}\left(1-\frac{1}{M_{n}+1}\right)+1-p_{n} \\
&=1-\frac{p_{n}}{M_{n}}+O(M_{n}^{-2}) 
=1-\frac{1}{n}f\left(\tfrac{M_n}{n}\right)+O(M_{n}^{-2}).
\end{align*}
Thus, by the monotonicity of $f(x)$ there is a $\delta>0$ such that 
\[
\Exp\left(\frac{1}{M_{n+1}} \mid \filt_{n} \right)<\frac{(1-(1+\delta/2)/n)}{M_{n}},
\] 
provided $n \geq n_0$ for some large $n_0$ and $n \leq \phi_1 \wedge \eta_C.$  Setting $C_{n+1}=(1+(1+\delta)/n)C_{n},$ $n>n_{0}$, we have that $A_n = C_n/M_n$ is a supermartingale for this same range of $n.$ 
By Lemma~\ref{lem:numbers} we have that $C_nn^{-1-\delta}$ converges to a positive limit, and   
by Doob's theorem $A_{n\wedge\phi_{1}\wedge \eta_C}$ tends to a finite limit with probability 1.
Thus there is a random constant $B > 0$ so that $M_n \geq B n^{1+\delta}$ for all $n \leq \phi_1 \wedge \eta_C.$  On the other hand, $M_n \leq 2n,$ and so it must be that $\phi_1 \wedge \eta_C < \infty$ almost surely.  Thus, on the event that $\eta_C = \infty,$ we have $\phi_1 < \infty,$ and so we can put $n_{1}=\phi_{1}$.

Now we turn to the upper bound, which proceeds by nearly the same argument, though using a different supermartingale.
To that end, consider the expectation:
\[
\Exp\left(\frac{M_{n+1}}{M_{n}} \mid \filt_{n}\right)
=\frac{p_n(M_{n}+1)}{M_{n}}+1-p_{n}
=1+\frac{p_{n}}{M_{n}}.
\]
Assume that for $n_{0},$ $x_{\ast}+\epsilon<M_{n_{0}}/n_{0}$ (otherwise we could just put $n_{2}=n_{0}$). Let $\phi_{2}$ be the first moment after $n_{0}$, such that $x_{\ast}+\epsilon\geq M_{\phi_{2}}/\phi_{2}$. We need to prove that on the event $\eta_C = \infty,$ $\phi_{2}<\infty$. 

Lemma \ref{lem:improved_low_bound2} and the monotonicity of $f(x)$ imply that if $x_{\ast}+\epsilon<M_{n}/n$ and if $n$ is large enough, then there is a $\delta>0$ such that $\Exp(M_{n+1}|\filt_{n})<(1+(1-\delta)/n)M_{n}$. Therefore $M_{n}/C_{n}$ is supermartingale for $n_{0}\leq n<\phi_{2}$, where $C_{n+1}=(1+(1-\delta)/n)C_{n},$ $n>n_{0}$. By Lemma~\ref{lem:numbers} we have that $C_nn^{-1+\delta}$ converges to a positive limit.  Setting $A_n = M_n/C_n,$ we have by Doob's theorem $A_{n\wedge\phi_{2}\wedge \eta_C}$ tends to a finite limit with probability 1, and in particular, there is a random constant $B > 0$ so that $M_n \leq B n^{1-\delta}$

However, for $n \leq \phi_2,$ we have that $M_n > (x_{\ast} + \epsilon)n,$ and so it must be that $\phi_2 \wedge \eta_C < \infty.$ Thus conditional on $\eta_C= \infty,$ we have $\phi_2 < \infty,$ which completes the proof.
\end{proof}

Now we need an auxiliary lemma about the sum of independent variables.
\begin{lemma}
\label{lem:sum_est}
Let $S_n$ denote a random walk with independent centered increments bounded by $1.$  For any $\alpha > 0$ there is a $c > 0$ so that for any $m \geq 0$
\[
\Pr \left[
\exists~n : S_n > \alpha n + m
\right] \leq ce^{-\alpha m/c}.
\]
\end{lemma}
\begin{proof}
For a fixed $n$, we have by Hoeffding's inequality that there is a $c > 0$ so that
\[
\Pr \left[
S_n > \alpha n + m
\right] \leq \exp(-c (\alpha n + m)^2/n) \leq \exp(-c\alpha^2 n) \exp(-2c\alpha m).
\]
Summing this over $n,$ we get
\[
\Pr \left[
\exists~n : S_n > \alpha n + m
\right] \leq \frac{\exp(-2c\alpha m)}{1-\exp(-c\alpha^2)},
\]
so that adjusting $c,$ we have the desired bound.
\end{proof}

Using this lemma we will prove next result:
\begin{lemma}
\label{lem:final_est}
With probability $1,$
\(
M_n/n \to x_\ast.
\)
\end{lemma}
\begin{proof}

We will show that for each $\epsilon > 0,$ $M_n/n > x_\ast + \epsilon$ only finitely many times with probability $1$.  The argument to show that it is less than $x_\ast - \epsilon$ only finitely many times is identical.  Together, both statements complete the proof.
For any $\epsilon >0,$ let $\INE$ denote the interval $((x_\ast + \epsilon/2), (x_\ast + 3\epsilon/4)).$
For any $n,$ let $\tau_n$ be the first time greater than $n$ that $M_n/n < x_\ast + \epsilon/4$ or $M_n/n > x_\ast + \epsilon.$  Call $\mathcal{A}_n$ the event
\[
\mathcal{A}_n =
\{
M_n/n \in \INE, M_{\tau_n} > \tau_n(x_\ast + \epsilon), \tau_n < \eta_n
\}
\]
As with probability $1,$ there is an $N$ so that $\eta_N = \infty,$
then by virtue of Lemma~\ref{lem:interval_est}, $M_n/n$ is larger than $x_\ast + \epsilon$ infinitely often if and only if $\mathcal{A}_n$ occurs infinitely often.

Set $q(x) = 1-(1-x/2)^d.$  From the monotonicity of $q$ and the definition of $x_\ast,$ we have that $q(x) < x_\ast$ for $x > x_\ast$ and $q(x) > x_\ast$ for $x < x_\ast.$ In particular, we have that
\[
\inf_{ x \in \INE } |q(x) - x| = \alpha > 0.
\]
Now, given that $M_n/n \in \INE,$ then for any $k \leq \tau_n \wedge \eta_n$ with $k \geq n,$ we have that
\[
\Pr\left[
M_{k+1} = M_k + 1 \mid \filt_k
\right] = q(M_k/k) \leq x_\ast + \epsilon - \alpha.
\]
Thus, $(M_{n+i})_{i=0}^{\tau_n \wedge \eta_n}$ is dominated from above by a simple random walk $S_i$ with constant drift $x_\ast + \epsilon - \alpha.$ It follows that we have the bound
\begin{align*}
\Pr\left[
M_{\tau_n}/\tau_n > x_\ast + \epsilon, \tau_n \leq \eta_n \mid M_n/n \in \INE
\right] 
\hspace{-1in}&\hspace{1in} \\
&\leq
\Pr\left[
\exists~i~: S_i > (n+i)(x_\ast + \epsilon) \mid S_0/n \in \INE
\right].
\end{align*}
We now write $\tilde S_i = S_i - S_0 - (x_\ast + \epsilon - \alpha)i,$ a simple random walk without drift started from $0.$  Applying Lemma~\ref{lem:sum_est} we get that
\begin{align*}
\Pr\left[
\exists~i~: S_i > (n+i)(x_\ast + \epsilon) \mid S_0/n \in \INE
\right]
\hspace{-1in}&\hspace{1in} \\
&=\Pr\left[
\exists~i~: \tilde S_i > (n(x_\ast + \epsilon) - S_0) + \alpha i \mid S_0/n \in \INE
\right] \\
&\leq e^{-c\alpha n (x_\ast + \epsilon)},
\end{align*}
for all $n$ sufficiently large.  Thus, applying Borel-Cantelli, we get that
\[
\Pr \left[
\mathcal{A}_n~\text{i.o.}
\right] = 0.
\]
The same argument shows that $M_n/n$ is not below $x_\ast - \epsilon$ infinitely often, completing the proof.
\end{proof}
\section*{Acknowledgements.}
The authors are grateful to Professor Itai Benjamini for helpful conversations.
\bibliographystyle{alpha}
\bibliography{po2c_pa}

\end{document}